\documentclass{amsart}%
\usepackage{amsfonts}
\usepackage{amsmath}
\usepackage{amssymb}
\usepackage{dsfont}
\usepackage{graphicx}%
\providecommand{\U}[1]{\protect\rule{.1in}{.1in}}
\newtheorem{theorem}{Theorem}
\theoremstyle{plain}

\newtheorem{corollary}{Corollary}

\newtheorem{example}{Example}

\newtheorem{lemma}{Lemma}

\newtheorem{problem}{Problem}

\newtheorem{remark}{Remark}

\numberwithin{equation}{section}
\begin{document}
\title[Reflexivity of Banach $C(K)$-modules]{Reflexivity of Banach $C(K)$-modules via the reflexivity of Banach lattices}
\author{Arkady Kitover}
\address{Department of Mathematics, Community College of Philadelphia, 1700 Spring
Garden Street, Philadelphia, PA 19130 }
\email{akitover@ccp.edu}
\urladdr{}
\author{Mehmet Orhon}
\address{Department of Mathematics and Statistics, University of New Hampshire, Durham,
NH 03824}
\email{mo@unh.ed}
\date{\today}

\begin{abstract}
We extend the well known criteria of reflexivity of Banach lattices due to
Lozanovsky and Lotz to the class of finitely generated Banach $C(K)$-modules.
Namely we prove that a finitely generated Banach $C(K)$-module is reflexive if
and only if it does not contain any subspace isomorphic to either $l^{1}$ or
$c_{0}$.

\end{abstract}
\maketitle


\section{Introduction}

Suppose $K$ is a compact Hausdorff space and $C(K)$ is the algebra of complex
(or real) valued functions on $K$ with the supremum norm. Let $X$ be a Banach
space and let $\mathcal{L}(X)$ denote the algebra of all bounded linear
operators on $X$ with the operator norm. Let $m:C(K)\rightarrow\mathcal{L}(X)$
be a contractive and unital algebra homomorphism. To assume that $m$ is
contractive is not a loss of generality. When $m$ is bounded, we can make $m$
contractive by passing to an equivalent norm on $X$ \cite[V.3.3, p. 361]{Sch}.
In general the kernel of $m$ is a closed ideal of $C(K)$. Therefore, by
reducing to the quotient $C(K)/Ker(m)$, we may assume that $m$ is one-to-one.
Then we will regard $X$ as a Banach $C(K)$-module with $ax=m(a)(x)$ for all
$a\in C(K)$ and $x\in X.$ When $x\in X,$ we denote by $X(x):=cl(C(K)x)$ the
\textbf{cyclic subspace} of $X$ generated by $x$. (Here `$cl$' denotes closure
in norm in $X$.) It is familiar that $X(x)$ is representable as a Banach
lattice where the cone $X(x)_{+}=cl(C(K)_{+}x)$ and $x$ is a quasi-interior
point of $X(x)_{+}$ \cite[V.3]{Sch}, \cite{Ve}, \cite[4.6, p. 22]{AAK}. (Here
$C(K)_{+}$denotes the non-negative continuous functions on $C(K).$) Note that
$m$ is contractive implies that $\Vert ax\Vert\leq\Vert bx\Vert$ whenever
$|a|\leq|b|$ for some $a,b\in C(K)$ and for any $x\in X$. This in turn induces
a Banach lattice norm on a cyclic subspace $X(x)$ without the need to pass to
an equivalent norm. More generally by $X(x_{1},x_{2},\ldots,x_{n})$ we denote
the closed submodule of $X$ generated by the subset $\left\{  x_{1}%
,x_{2},\ldots,x_{n}\right\}  $. If for such a subset we have $X=X(x_{1}%
,x_{2},\ldots,x_{n})$, we will say that $X$ is \textbf{finitely generated}. It
is a well known result of Lozanovsky \cite{Loz}, \cite{Sch}, \cite[2.4.15, p.
94]{MN} that a Banach lattice is reflexive if and only if it does not contain
a copy of either $l^{1}$ or $c_{0}.$ We say a Banach space $X$ contains a copy
of a Banach space $Y,$ if there is a subspace of $X$ that is isomorphic to
$Y.$ In this paper our purpose is to extend the result of Lozanovsky to
finitely generated Banach $C(K)$-modules. Namely, a finitely generated Banach
$C(K)$-module $X$ is reflexive if and only if $X$ does not contain a copy of
either $l^{1}$ or $c_{0}.$

Concerning the reflexivity of Banach lattices, there is a result due to Lotz
\cite{Lot}, \cite{Sch},\cite[2.4.15, p. 94]{MN} that in one direction is
stronger than the Theorem of Lozanovsky. Namely, Lotz showed that a Banach
lattice is reflexive if and only if it does not contain a copy of either
$l^{1}$ or $c_{0}$ as a sublattice. That is, no sublattice of the Banach
lattice is lattice isomorphic to $l^{1}$ or $c_{0}$.\ In the proof of Theorem
\ref{t1} we will use the Theorem of Lotz as well.

\begin{theorem}
\label{t1}Let $X$ be a finitely generated Banach $C(K)$-module. Then the
following are equivalent:

\begin{enumerate}
\item $X$ is reflexive,

\item $X$ does not contain a copy of either $l^{1}$ or $c_{0}$,

\item $X^{\prime}$ does not contain a copy of $l^{1}$.

\item Each cyclic subspace of $X$ does not contain a copy of either $l^{1}$ or
$c_{0}$,

\item Each cyclic subspace of $X$ is reflexive.
\end{enumerate}
\end{theorem}

Since one may regard any Banach space as a Banach module over $\mathbb{C}$ (or
$\mathbb{R}$), the well known example, the James Space \cite{J}, \cite[1.d.2,
p. 25]{LiT}, \cite[5.17, p. 325]{MN}, or its variants show that in general we
cannot drop the condition that the module is finitely generated.

\section{Preliminaries}

Initially we will recall some information concerning Banach $C(K)$-modules
that we need. Suppose that the compact Hausdorff space $K$ is Stonian in the
sense that the closure of an open set in $K$ is open (i.e., clopen). Then the
characteristic functions of the clopen sets (i.e., the idempotents in $C(K)$)
form a complete Boolean algebra of projections $\mathcal{B}$ on $X$ and the
closed linear span of $\mathcal{B}$ in $\mathcal{L}(X)$ is equal to $m(C(K)).$
A stronger condition is to require that $\mathcal{B}$ is a \emph{Bade complete
}Boolean algebra of projections on $X$ in the sense that in addition to being
a complete Boolean algebra, $\mathcal{B}$ has the property that whenever
$\{\chi_{\alpha}:\alpha\in\Gamma\}$ is an increasing net of idempotents in
$\mathcal{B}$ with least upper bound $\chi\in\mathcal{B}$ and $x\in X,$ then
$\chi_{\alpha}x$ converges to $\chi x$ in $X$ \cite{B1}, \cite[XVII.3.4, p.
2197 ]{DS}, \cite[V.3, p. 315]{Sch}. This is equivalent to that $K$ is
hyperstonian (i.e., $C(K)$ is a dual Banach space) and that the homomorphism
$m$ is continuous with respect to the weak*-topology on $C(K)$ and the weak
operator topology on $\mathcal{L}(X)$ \cite[Theorem 1]{O1}. It also implies
that $m(C(K))$ is closed in the weak operator topology in $\mathcal{L}(X)$
\cite[XVII.3.17, p. 2213]{DS}. In such a case each cyclic subspace $X(x)$, as
a Banach lattice, has order continuous norm \cite[V.3.6, p. 318]{Sch},
\cite{Ve} and its ideal center $Z(X(x))$ is given by $m(C(K))_{|X(x)}$
\cite[6.2.3, p. 35]{AAK}, \cite[Theorem 1]{O3}, that is, each operator in
$m(C(K))$ is restricted to the subspace $X(x)$.

For a general Banach $C(K)$-module $X,$ the weak operator closure of $m(C(K))$
in $\mathcal{L}(X)$ is given by the range of a map $\widehat{m}:C(\widehat
{K})\rightarrow\mathcal{L}(X)$ where $\widehat{K}$ is a compact Hausdorff
space that contains $K$ as a quotient and $\widehat{m}$ is a contractive,
unital and one-to-one algebra homomorphism that extends $m$ \cite[6.3, p.
35]{AAK}, \cite[Corollary 8]{HO}. The closed submodules of $X$ are the same
with respect to either module structure therefore we may assume that $m(C(K))$
is weak operator closed without any loss of generality. The weak operator
closure of $m(C(K))$ is generated by a Bade complete Boolean algebra of
projections if and only if $m(C(K))$ has weakly compact action on $X$ in the
sense that the mapping $C(K)\rightarrow X$ defined by $a\rightarrow ax$ for
all $a\in C(K)$ for a fixed $x\in X$ is a weakly compact linear map for each
$x\in X$ \cite[Theorem 3]{O1}. For example if $X$ does not contain any copy of
$c_{0},$ then $m(C(K))$ has weakly compact action on $X$ \cite{Pel}. In
general even when $m(C(K))$ is weak operator closed in $\mathcal{L}(X),$ when
restricted to a closed submodule $Y,$ $m(C(K))_{|Y}$ need not be weak operator
closed in $\mathcal{L}(Y)$. An exception is when the weak operator closed
algebra $m(C(K))$ is generated by a Bade complete Boolean algebra of
projections on $X.$ In this case both $m(C(K))|_{Y}$ and $m(C(K))|_{X/Y}$ are
generated by a Bade complete Boolean algebra of projections.

\begin{lemma}
\label{L1}Suppose that $X$ is a Banach $C(K)$-module such that $m(C(K))$ is
generated by a Bade complete Boolean algebra of projections on $X$. Let $Y$ be
a closed submodule of $X.$ Then, when restricted to either $Y$ or $X/Y$,
$m(C(K))$ is generated by a Bade complete Boolean algebra of projections.
\end{lemma}

\begin{proof}
Since $m(C(K))$ is generated by a Bade complete Boolean algebra of projections
on $X,$ we have that $K$ is hyperstonian. The Boolean algebra of the
idempotents $\mathcal{B}$ in $C(K)$ is the Bade complete Boolean algebra of
projections that generates $m(C(K)).$ That is whenever $\{\chi_{\alpha}%
:\alpha\in\Gamma\}$ is an increasing set of idempotents in $\mathcal{B}$ with
least upper bound $\chi\in\mathcal{B}$ and $x\in X,$ then $\chi_{\alpha}x$
converges to $\chi x$ in $X.$ In general $m$ will not be one-to-one when
restricted to $Y$ or $X/Y.$ However that $m$ is weak* to weak-operator
continuous implies that the kernel of $m$ in both cases will be a weak*-closed
ideal in $C(K).$ Here we will give the proof of the lemma in the case of
$X/Y$. In the case of $Y$, a proof is given in \cite[XVII.11, p. 2204 and
XVII.23 p. 2215]{DS}, but a proof similar to the one below is also possible.
Let $\{a_{\alpha}\}$ be a net in $C(K)$ that converges to $a\in C(K)$ in the
weak*-topology. Then $m$ is weak* to weak-operator continuous implies that
$\{a_{\alpha}x\}$ converges to $ax$ weakly for each $x\in X$. Suppose that
$\{a_{\alpha}\}$ is in the kernel of $m$ when restricted to $X/Y.$ Then
$a_{\alpha}x\in Y$ for all $x\in X$ and for all $\alpha$ in the index set.
Since $Y$ is weakly closed, we have $ax\in Y.$ So the kernel of $m$ is a
weak*-closed ideal and therefore a weak*-closed band in $C(K).$ That is there
exists an idempotent $\pi\in C(K)$ such that $m$ is one-to-one on $\pi C(K)$
and the kernel of $m$ is $(1-\pi)C(K)$. Clearly $\pi C(K)=C(S)$ for some
clopen subset $S$ of $K$ and furthermore, since $\pi$ is a weak*-continuous
band projection on $C(K)$, $C(S)$ is also a dual Banach space. For any $x\in
X,$ let $[x]=x+Y$ in $X/Y$. Then for the increasing net of idempotents above
we have
\[
||(\chi-\chi_{\alpha})[x]||\leq||(\chi-\chi_{\alpha})x||
\]
where the right hand side of the inequality goes to $0$ in $X.$ Hence
$\pi\mathcal{B}$ is a Bade complete Boolean algebra of projections on $X/Y.$
\end{proof}

We also need some lemmas concerning Banach spaces.

\begin{lemma}
\label{L2}Let $X$ be a Banach space and $Y$ be a reflexive subspace of $X.$
For any $x\in X$, let $[x]=x+Y$ in $X/Y$.

\begin{enumerate}
\item For any $x\in X$ there is $y\in Y$ such that $||[x]||=||x+y||$.

\item If $||[x]||\leq C$ for some $x\in X$, then $G=\{y\in Y:||x+y||\leq C\}$
is a non-empty, convex and weakly compact subset of $Y.$
\end{enumerate}
\end{lemma}

\begin{proof}
Let $x\in X\setminus Y$ and $n$ be a positive integer. There is $y_{n}\in Y$
such that $||x+y_{n}||<||[x]||+\frac{1}{n}.$ Then $||y_{n}||< 2||x||+\frac
{1}{n}$ for each $n=1,2\ldots.$ That is $\{y_{n}\}$ is a bounded sequence in
the reflexive space $Y.$ Therefore it has a weakly convergent subsequence.
Without loss of generality assume that $\{y_{n}\}$ converges weakly to some
$y\in Y.$ Let $f\in X^{\prime}$ with $||f||=1.$ Given any $\varepsilon>0,$ we
have
\[
|f(x+y)-f(x+y_{n})|=|f(y-y_{n})|<\varepsilon
\]
for sufficiently large $n.$ Hence, for sufficiently large $n,$%
\[
|f(x+y)|<||x+y_{n}||+\varepsilon<||[x]||+\varepsilon+\frac{1}{n}.
\]

Therefore $||x+y||\leq||[x]||+\varepsilon$ for all $\varepsilon>0.$ Then
$||x+y||=||[x]||$ and $||y||\leq||x||+C.$ From this it follows that $G$ is
non-empty, bounded, and closed. It is easy to check that $G$ is convex. Since
$Y$ is reflexive $G$ is weakly compact.
\end{proof}

\begin{lemma}
\label{L3}Let $X$ be a Banach space such that its dual $X^{\prime}$ does not
contain any copy of $l^{1}.$ Then $X$ does not contain any copy of either
$l^{1}$ or $c_{0}.$
\end{lemma}

\begin{proof}
The proof depends on two well known main results of this topic. A result of
Bessaga and Pelczynski states that if $X^{\prime}$ contains a copy of $c_{0}$
then $X$ contains a complemented copy of $l^{1}$ \cite{BP},\cite[2.e.8, p.
103]{LiT}. Since we assume that $X^{\prime}$ does not contain any copy of
$l^{1},$ $X^{\prime\prime}$ does not contain any copy of $c_{0}.$ Then, as a
subspace of $X^{\prime\prime},$ $X$ does not contain any copy of $c_{0}.$ On
the other hand, a result of Hagler states that $X$ contains a copy of $l^{1}$
if and only if $X^{\prime}$ contains a copy of $L^{1}[0,1]$ \cite{Hag}. Since
$L^{1}[0,1]$ contains copies of $l^{1}$, our assumption on $X^{\prime}$
implies that $X^{\prime}$ does not contain any copy of $L^{1}[0,1]$ and
therefore $X$ does not contain any copy of $l^{1}$.
\end{proof}

In the proofs of Theorem \ref{t1} and Theorem \ref{t2} we will use ultranets.
We will briefly review the definition of an ultranet and some basic facts that
we will use concerning ultranets. A net $\{x_{\lambda}\}$ in a set $F$ is
called an ultranet if for each subset $G$ of $F$, $\{x_{\lambda}\}$ is either
eventually in $G$ or eventually in $F\diagdown G$ \cite[11.10, p. 76]{W}. If
$\{x_{\lambda}\}$ is an ultranet and $f:F\rightarrow H$ is a function then
$\{f(x_{\lambda})\}$ is an ultranet in $H$ \cite[11.11, p. 76]{W}. Every net
has a subnet that is an ultranet and every subnet of an ultranet is an
ultranet \cite[11B, p. 77]{W}. Finally, every ultranet in a compact Hausdorff
space converges \cite[17.4, p.118]{W}.

\section{Proof of the main result}

We are now ready to give a proof of Theorem \ref{t1}.

\begin{proof}
When $X$ is reflexive then $X^{\prime}$ is reflexive and therefore $X^{\prime
}$ can not contain any copy of $l^{1}.$ That is (1) implies (3). Lemma
\ref{L3} yields that (3) implies (2). It is clear that (2) implies (4). Also,
since for each $x\in X$\ \ the cyclic subspace $X(x)$ may be represented as a
Banach lattice, Lozanovsky's Theorem implies that (4) if and only if (5).
Hence the proof will be complete if we show (4) implies (1).

Initially observe that by (4) each cyclic subspace of $X$ does not contain any
copy of $c_{0}.$ Then, as noted in Section 2, $m(C(K))$ restricted to the
cyclic subspace has its weak operator closure generated by a Bade complete
Boolean algebra of projections. Hence for each $x\in X$, the map
$m(C(K))\longrightarrow X$ defined by $a\rightarrow ax$ is a weakly compact
operator. Therefore, if we assume without loss of generality that $m(C(X))$ is
weak operator closed in $L(X)$, then, again as noted in Section 2, we have
that $m(C(K))$ is generated by a Bade complete Boolean algebra of projections.
This means that $K$ is hyperstonian and the idempotents in $C(K)$ (which
correspond to the characteristic functions of the clopen subsets of $K$ ) form
the Bade complete Boolean algebra of projections that generate $m(C(K))$ on
$X.$

Now we prove (4) implies (1) by induction on the number of generators of $X.$
Suppose $X$ is generated by one element $x_{0}\in X$. Then the cyclic space
$X=X(x_{0})$ may be represented as a Banach lattice. Hence, (4) and
Lozanovsky's Theorem imply that $X$ is reflexive. Now suppose that whenever a
finitely generated Banach $C(K)$-module has $r\geq1$ generators and satisfies
(4) then it is reflexive. Suppose $X$ is a Banach $C(K)$-module with $r+1$
generators and satisfies (4). Let $\{x_{0},x_{1},\ldots,x_{r}\}$ be a set of
generators of $X.$ Let $Y$ be the closed submodule of $X$ generated by
$\{x_{1},x_{2}\ldots,x_{r}\}.$ Then $Y$ satisfies (4) and therefore $Y$ is
reflexive by the induction hypothesis. Since we have that $m(C(K))$ is
generated by a Bade complete Boolean algebra of projections on $X$, Lemma
\ref{L1} implies that the same is true for the quotient $X/Y.$ Note that since
$X/Y$ is a cyclic space generated by $[x_{0}]$, it may be represented as a
Banach lattice such that $[x_{0}]$ is a quasi-interior point and the ideal
center $Z(X/Y)=m(C(K))_{|X/Y}$. That $m(C(K))$ is generated by a Bade complete
Boolean algebra of projections on $X/Y$ implies that as a Banach lattice $X/Y$
has order continuous norm. This means in particular that each band in $X/Y$ is
the range of a band projection. Suppose the Banach lattice $X/Y$ is not
reflexive then, by Lotz's Theorem , $X/Y$ must contain a copy of either
$l^{1}$ or $c_{0}$ as a sublattice \cite{Lot}.

First, assume that there is a sublattice of $X/Y$ that is lattice isomorphic
to $l^{1}.$ Let $\{e_{n}\}$ be a sequence in $X$ such that $\{[e_{n}]\}$
corresponds to the basic sequence of $l^{1}$ in the sublattice of $X/Y$ that
is lattice isomorphic to $l^{1}.$ That is $\{[e_{n}]\}$ is a pairwise disjoint
positive sequence in $X/Y$ such that for some $0<d<D$ we have for each
$(\xi_{n})\in l^{1}$%
\[
d%
{\textstyle\sum}
|\xi_{n}|\leq\Vert%
{\textstyle\sum}
\xi_{n}[e_{n}]\Vert\leq D%
{\textstyle\sum}
|\xi_{n}|.
\]
Here $%
{\textstyle\sum}
\xi_{n}[e_{n}]$ represents the limit of the Cauchy sequence given by the
partial sums of the series in $X/Y.$ Let $\chi_{n}\in C(K)$ be the band
projection onto the band generated by $[e_{n}]$ in $X/Y$. Since the elements
of the sequence $\{[e_{n}]\}$ are disjoint in $X/Y$, the elements of the
sequence of band projections $\{\chi_{n}\}$ are also disjoint as idempotents
in $C(K)$. Since $[e_{n}]=\chi_{n}[e_{n}]=[\chi_{n}e_{n}],$ without loss of
generality, we assume that $e_{n}=\chi_{n}e_{n}$ for each $n$. Furthermore
again without loss of generality we may assume that for some $\varepsilon>0$
and for each $n,$ we have $\Vert e_{n}\Vert\leq D(1+\varepsilon)$ in $X$. Now
we have for each $(\xi_{n})\in l^{1}$ and positive integer $N$,%
\[
d%
{\textstyle\sum_{n=1}^{N}}
|\xi_{n}|\leq\Vert%
{\textstyle\sum_{n=1}^{N}}
\xi_{n}[e_{n}]\Vert\leq\Vert%
{\textstyle\sum_{n=1}^{N}}
\xi_{n}e_{n}\Vert\leq%
{\textstyle\sum_{n=1}^{N}}
|\xi_{n}|\Vert e_{n}\Vert\leq D(1+\varepsilon)%
{\textstyle\sum}
|\xi_{n}|.
\]
By passing to the limit, we have%
\[
d%
{\textstyle\sum}
|\xi_{n}|\leq\Vert%
{\textstyle\sum}
\xi_{n}e_{n}\Vert\leq D(1+\varepsilon)%
{\textstyle\sum}
|\xi_{n}|
\]
where $%
{\textstyle\sum}
\xi_{n}e_{n}$ is now the limit of the Cauchy sequence given by the partial
sums of the series in $X.$ Hence $X$ has a subspace that is isomorphic to
$l^{1}.$ Let $y=%
{\textstyle\sum}
\frac{1}{2^{n}}e_{n}.$ Then $\chi_{n}y=$ $\frac{1}{2^{n}}e_{n}$ for each $n.$
That is, the subspace of $X$ that is isomorphic to $l^{1}$ is contained in the
cyclic subspace $X(y)$ of $X.$ This contradicts (4) and thus $X/Y$ does not
contain a copy of $l^{1}$ as a sublattice.

It follows that if the Banach lattice $X/Y$ is not reflexive, it must contain
a copy of $c_{0}$ as a sublattice. Let $\{e_{n}\}$ be a sequence in $X$ such
that $\{[e_{n}]\}$ corresponds to the basic sequence of $c_{0}$ in the
sublattice of $X/Y$ that is lattice isomorphic to $c_{0}.$ Let $\chi_{n}\in
C(K)$ be the band projection onto the band generated by $[e_{n}]$ in $X/Y$.
Since the elements of the sequence $\{[e_{n}]\}$ are disjoint in $X/Y$, the
elements of the sequence of band projections $\{\chi_{n}\}$ are also disjoint
as idempotents in $C(K)$. Since $[e_{n}]=\chi_{n}[e_{n}]=[\chi_{n}e_{n}],$
without loss of generality, we assume that $e_{n}=\chi_{n}e_{n}$ for each $n$.
We will assume that $0<d<D$ are the constants that give the lattice
isomorphism of $c_{0}$ into $X/Y$. That is
\[
d(\sup|\xi_{n}|)\leq||\sum\xi_{n}[e_{n}]||\leq D(\sup|\xi_{n}|)
\]
when $(\xi_{n})\in c_{0}$ where $\sum\xi_{n}[e_{n}]$ denotes the limit of the
Cauchy sequence given by the partial sums of the series in $X/Y$. Let
$z_{n}=e_{1}+e_{2}\ldots+e_{n}$ and $\zeta_{n}=\chi_{1}+\chi_{2}+\ldots
+\chi_{n}$ for each $n=1,2,\ldots$. Let $G_{n}=\{y\in Y:||z_{n}+y||\leq D\}$
for each $n.$ By Lemma \ref{L2}, $G_{n}$ is a non-empty, convex and weakly
compact subset of $Y.$ Let $H_{n}=\zeta_{n}G_{n}.$ Clearly, when $n\geq m$,
$\zeta_{m}z_{n}=z_{m}.$ Then, for each $y\in G_{n},$
\[
||z_{n}+\zeta_{n}y||=||\zeta_{n}(z_{n}+y)||\leq||z_{n}+y||\leq D.
\]
It follows that, for each $n,$ $H_{n}\subset G_{n}$ and that $H_{n}$ is itself
a non-empty, convex and weakly compact subset of $Y.$ We choose a sequence
$\{y_{i}\}$ in $Y$ with $y_{i}\in H_{i}.$ It follows from above that for each
$n,$ the sequence $\{\zeta_{n}y_{i}\}_{i\geq n}^{\infty}$ is in $H_{n}.$
Namely, if $i\geq n$ then%
\[
||z_{n}+\zeta_{n}y_{i}||=||\zeta_{n}(z_{i}+y_{i})||\leq||z_{i}+y_{i}||\leq D.
\]
Now let $\{i_{\alpha}\}_{\alpha\in\Gamma}$ be an ultranet that is a subnet of
the sequence of positive integers $\{i\}$ and for each $n$ let $\Gamma
_{n}=\{\alpha\in\Gamma:i_{\alpha}\geq n\}.$ Then $\{\zeta_{n}y_{i_{\alpha}%
}\}_{\alpha\in\Gamma_{n}}$ is an ultranet that is a subnet of the sequence
$\{\zeta_{n}y_{i}\}_{i\geq n}^{\infty}$ in $H_{n}.$ Since $H_{n}$ is weakly
compact, the ultranet $\{\zeta_{n}y_{i_{\alpha}}\}_{\alpha\in\Gamma_{n}}$
converges weakly to some $w_{n}\in H_{n}.$ Since for any positive integers
with $n$ $\geq m$, we have $\Gamma_{n}\subset\Gamma_{m}$ it follows from the
definition of the sequences that $\zeta_{m}w_{n}=w_{m}$ whenever $n\geq m.$ In
particular, since $\zeta_{n-1}w_{n}=w_{n-1},$ we have $w_{n}=\chi_{n}%
w_{n}+w_{n-1}$ for all $n\geq2.$ Hence, by induction, we have that $w_{n}%
=\chi_{1}w_{1}+\chi_{2}w_{2}\ldots+\chi_{n}w_{n}$ for all positive integers
$n.$ Define a sequence $\{u_{n}\}$ in $X$ such that $u_{n}=e_{n}+\chi_{n}%
w_{n}$ for each $n.$ Clearly
\[
d\leq||[e_{n}]||\leq||u_{n}||
\]
for each $n.$ Also, since $u_{1}+u_{2}\ldots+u_{n}=z_{n}+w_{n}$ we have%
\[
\Vert u_{n}\Vert\leq||u_{1}+u_{2}\ldots+u_{n}||=||z_{n}+w_{n}||\leq D
\]
for each $n.$ Then it follows that the closed subspace spanned by $\{u_{n}\}$
in $X$ is lattice isomorphic to $c_{0}.$ To see this, let $u=\sum
_{n=1}^{\infty}\frac{u_{n}}{2^{n}}$ in $X$ and consider the cyclic subspace
$X(u)$ of $X.$ We have that $X(u)$ is a Banach lattice with quasi-interior
point $u$ and ideal center $Z(X(u))=m(C(K))_{|X(u)}$. Since $\chi_{n}%
u=\frac{u_{n}}{2^{n}}$ for each $n,$ $\{u_{n}\}$ is a pairwise disjoint
positive sequence in the Banach lattice $X(u).$ Then the closed sublattice of
$X(u)$ generated by $\{u_{n}\}$ is lattice isomorphic to $c_{0}$ , this
follows directly from the $Z(X(u))$ module structure of $X(u)$, but also it
follows by \cite[ 2.3.10, p. 82]{MN}. This contradicts the assumption (4).
Hence, the Banach lattice $X/Y$ cannot contain a copy of $c_{0}$ as a
sublattice, as well as not containing a copy of $l^{1}$ as a sublattice. Then
Lotz's refinement of Lozanovsky's Theorem implies that $X/Y$ is reflexive.
Since by induction hypothesis $Y$ is reflexive, we have that $X$ is also
reflexive. The reader will observe that this is the familiar three space
property of reflexivity. Therefore (4) implies (1).
\end{proof}

One of the questions naturally connected with the statement of
Theorem~\ref{t1} is the following one. Can we substitute condition $(5)$ of
this theorem with a weaker condition that for some set $\{z_{1}, \ldots,
z_{n}\}$ of generators of the $C(K)$-module $X$ the cyclic subspaces
$X(z_{i}), i=1, \ldots, n$ are reflexive? The two examples below show that the
answer to this question in general is negative.

\begin{example}
\label{e1} Let $E=L^{2}(0,1)\oplus L^{1}(0,1)\oplus L^{1}(0,1)$ be an
$L^{\infty}(0,1)$-module with the norm $\Vert(f,g,h)\Vert=\Vert f\Vert
_{2}+\Vert g\Vert_{1}+\Vert h\Vert_{1}$ for all $f\in L^{2}(0,1)$ and $g,h\in
L^{1}(0,1)$. The module structure is carried over coordinatewise from the
usual $L^{\infty}(0,1)$-module structures of $L^{2}(0,1)$ and $L^{1}(0,1)$
that are given by almost everywhere pointwise multiplication.

Let $X = \{(f,g,h) \in E : f+g + h =0\}$. Then $X$ is a closed submodule of
$E$ that has two generators. It is easy to see for example that $\{(-1,1,0),
(1, 0, -1)\}$ generates $X$ where $1$ is the identity in $L^{\infty}(0,1)$.
Then $X((1,-1,0))$ and $X((1,0,-1))$ are both isomorphic to $L^{2}(0,1)$. Thus
the cyclic subspace generated by either of these vectors is reflexive.

But $X$ itself is not reflexive. Indeed $(0,1,-1) \in X$ and $X((0,1,-1))$ is
isomorphic to $L^{1}(0,1)$ and hence is not reflexive. Moreover, if we use the
set of generators $\{(1,0,-1), (0,1,-1)\}$, it is straightforward to see that
$X$ is isomorphic (but not isometric) to $L^{2}(0,1) \oplus L^{1}(0,1)$.
\end{example}

Now we will provide an even simpler example of a nonreflexive Banach
$C(K)$-module $X$ which is the direct sum of two cyclic subspaces but also has
two generators such that the corresponding cyclic subspaces are reflexive.

\begin{example}
\label{e2} Let $X=c_{0}\oplus l^{2}$ be a Banach $l^{\infty}$-module with the
norm $\Vert(x,y)\Vert=\Vert x\Vert_{\infty}+\Vert y\Vert_{2}$, for all $x\in
c_{0}$ and $y\in l^{2}$. Like in Example~\ref{e1} the $l^{\infty}$-module
structure is carried over coordinatewise from the module structure of $c_{0}$
and $l^{2}$. It is straightforward to observe that if $x_{n}=1/n,n\in
\mathds{N}$ then $x\in c_{0}\cap l^{2}$ and the vectors $(x,x)$ and $(0,x)$
generate $X$. Moreover, $X((x,x))$ is isomorphic to $l^{2}$ and
$X((0,x))=l^{2}$. Hence once again $X$ is not reflexive but has a pair of
generators such that each of them generates a reflexive cyclic subspace.
\end{example}

\section{Boolean algebras of projections of finite multiplicity}

In this section we will consider extending the conclusions of Theorem \ref{t1}
to the case where the Banach $C(K)$-module may be infinitely generated while
staying close to being finitely generated. Let $X$ be a Banach $C(K)$-module
where $K$ is hyperstonian and the homomorphism $m$ is weak* to weak-operator
continuous. That is, as discussed in Section 2, $\mathcal{B}$ (the set of
idempotents in $m(C(K))$) is a Bade complete Boolean algebra of projections on
$X$. \emph{Throughout this section we will assume that} $X$ \emph{is a Banach
space and} $\mathcal{B}$ \emph{is a Bade complete Boolean algebra of
projections on} $X$. Then $\mathcal{B}$ is said to be of \emph{uniform
multiplicity} $\emph{n}$ if there exists a disjoint family of projections
$\{\chi_{\alpha}\}$ in $\mathcal{B}$ such that for each projection $\chi
\in\mathcal{B}$ with $\chi\chi_{\alpha}=\chi$ one has $\chi X$ is generated by
a minimum of $n$ elements in $\chi X$ , and also $\sup\chi_{\alpha}=1$ in
$\mathcal{B}$ \cite{B2}, \cite[XVIII.3.1, p. 2264 and XVIII.3.6, p. 2267]{DS}.

If $\mathcal{B}$ is of uniform multiplicity one then it was shown by Rall
\cite{Rl}, that $X$ is represented as a Banach lattice with order continuous
norm and its ideal center is $m(C(K))$ (for a proof, see \cite[Lemma 2]{O2}).
The only difference from the cyclic case is that instead of a quasi-interior
element one has a topological orthogonal system \cite[III.5.1, p. 169]{Sch}.
(For example consider $l^{2}(\Gamma)$ when $\Gamma$ is an uncountable discrete
set.) To prove our next result, Corollary \ref{c1} , we need the following lemma.

\begin{lemma}
\label{l4}Let $X$ be a Banach space and let $\mathcal{B}$ be a Bade complete
Boolean algebra of projections on $X$ that is of uniform multiplicity one.
Then $X$ is reflexive if and only if no cyclic subspace of $X$ contains a copy
of either $l^{1}$ or $c_{0}$.
\end{lemma}

\begin{proof}
It is necessary to prove only one direction. Namely, suppose $X$ has no cyclic
subspace that contains a copy of either $l^{1}$ or $c_{0}$. By Rall's Theorem,
$X$ is represented as a Banach lattice with order continuous norm and
$\mathcal{B}$ corresponds to the Boolean algebra of band projections of the
Banach lattice $X$. Since $\mathcal{B}$ is of uniform multiplicity one on $X$,
it is well known that, for each $x\in X$, there exists $e_{x}\in\mathcal{B}$
such that $X(x)=e_{x}X$ (e.g., ~\cite[XVIII.3.6, p. 2267]{DS}). Thus $X(x)$ is
the band generated by $x$ in $X$. Suppose $X$ is not reflexive. Then by Lotz's
Theorem there exists a sublattice of $X$ that is lattice isomorphic either to
$l^{1}$ or $c_{0}$. Let $\{u_{n}\}$ be the norm bounded sequence of pairwise
disjoint elements of $X$ that corresponds to the basic sequence of either
$l^{1}$ or $c_{0}$ under this lattice isomorphism. Let $u=\sum\frac{u_{n}%
}{2^{n}}\in X$. Let $e_{n}$ be the band projection onto the band generated by
$u_{n}$ in $X$. Since the elements $u_{n}$ are disjoint, $e_{n}u=\frac{u_{n}%
}{2^{n}}$ for each $n$. Therefore the sublattice of $X$ generated by
$\{u_{n}\}$ is contained in $X(u)$. This contradicts the assumption that no
cyclic subspace of $X$ contains a copy of either $l^{1}$ or $c_{0}$.
\end{proof}

\begin{remark}
\label{r2} The special case of Lemma~\ref{l4} when $X$ is cyclic was obtained
by Tzafriri in~\cite{Tz2}.
\end{remark}

Then the methods of Theorem \ref{t1} yield the following corollary.

\begin{corollary}
\label{c1}Let $X\mathcal{\ }$be a Banach space and let $\mathcal{B}$ be a Bade
complete Boolean algebra of projections on $X$ that is of uniform multiplicity
$n$. Then the conditions (1)-(5) of Theorem \ref{t1} are equivalent.
\end{corollary}

\begin{proof}
The part of the proof up to (4) implies (1) is as in the proof of Theorem
\ref{t1} and is clear. In the proof of (4) implies (1), Lemma~\ref{l4} shows
the case $n=1$ is true. The rest of the proof follows by induction on $n$ just
as in the proof of Theorem \ref{t1}.
\end{proof}

We say $\mathcal{B}$ is of \emph{finite multiplicity} on $X$ if there exists a
disjoint family of projections $\{\chi_{\alpha}\}$ in $\mathcal{B}$ such that
$\chi_{\alpha}X$ is generated by a minimum of $n_{\alpha}$ elements in
$\chi_{\alpha}X$ and $\sup\chi_{\alpha}=1$ in $\mathcal{B}$. Then it follows
by a result of Bade that there exists a sequence $\{e_{n}\}$ of disjoint
projections in $\mathcal{B}$ such that $e_{n}\mathcal{B}$ is of uniform
multiplicity $n$ on $e_{n}X$ and $\sup e_{n}=1$ \cite[XVIII.3.8, p. 2267]{DS}.
In such a case it is clear that $X=cl(span\{e_{n}X:n=1,2,\ldots\})$.

In the proof of Theorem \ref{t2} we will need some additional properties of
the spaces involved that we have not used so far. In the discussion that
follows we will introduce these.

Given a Banach $C(K)$-module $X$, its dual $X^{\prime}$ has a natural Banach
$C(K)$-module structure induced by the module structure of $X$. Let
$m^{\prime}:C(K)\rightarrow\mathcal{L}(X^{\prime})$ be defined by $m^{\prime
}(a)=(m(a))^{\prime}$ (here the right hand side of the equality denotes the
adjoint of $m(a)$ in $\mathcal{L}(X^{\prime})$), for each $a\in C(K)$. It is
clear that $m^{\prime}$ is a contractive unital algebra homomorphism. When $m$
is one-to-one, $m^{\prime}$ is also one-to-one. Furthermore to distinguish
between the module action of $C(K)$ on $X^{\prime}$ and on $X$, we will define
$a^{\prime}f:=m^{\prime}(a)(f)$ for each $a\in C(K)$ and $f\in X^{\prime}$.
Then
\[
a^{\prime}f(x)=f(ax)
\]
for all $a\in C(K)$, $x\in X$, and $f\in X^{\prime}$. In particular for an
idempotent $\chi\in\mathcal{B}\subset C(K)$, we have%
\[
\chi^{\prime}f(x)=\chi^{\prime}f(\chi x)=f(\chi x)
\]
for all $x\in X,$ and $f\in X^{\prime}$.

In this section we assumed that $\mathcal{B}$ is a Bade complete Boolean
algebra of projections on $X$. Therefore $K$ is hyperstoinian. So in
particular $K$ is a Stonian compact Hausdorff space. Let $U$ be a dense open
subset of $K$ and let $f:U\rightarrow\mathbb{C}$ be a complex-valued, bounded
continuous function. Then $f$ has a unique extension $\widetilde{f}\in C(K)$
such that $\Vert\widetilde{f}\Vert=\underset{t\in U}{\sup}|f(t)|$ \cite[15G,
p. 106]{W}. Suppose $\{\zeta_{n}\}$ is a pairwise disjoint sequence of
idempotents in $\mathcal{B}$ such that $\sup\zeta_{n}=1$. Then we can embed
$l^{\infty}$ into $C(K)$ by means of an isometric unital algebra homomorphism.
To see this let $U=\{t\in K:\zeta_{n}(t)=1$ for some $n\}$. Then $U$ is an
open dense subset of $K$. Suppose $(a_{n})\in l^{\infty}$ and define
$f:U\rightarrow\mathbb{C}$ by%
\[
f(t)=a_{n}%
\]
whenever $\zeta_{n}(t)=1$ for some $t\in U$. Clearly $f$ is a bounded
continuous function. We will denote its unique extension in $C(K)$ by $%
{\textstyle\sum}
a_{n}\zeta_{n}$. It is clear that $\Vert%
{\textstyle\sum}
a_{n}\zeta_{n}\Vert=\Vert(a_{n})\Vert_{\infty}$ and the map $(a_{n}%
)\rightarrow%
{\textstyle\sum}
a_{n}\zeta_{n}$ gives the embedding of $l^{\infty}$ in $C(K)$ with the
required properties.

Finally we recall that a Banach lattice $X$ is called a KB-space, if any norm
bounded increasing non-negative sequence in $X$ has a least upper bound and
the sequence converges to its least upper bound in norm. It is well known that
$X$ is a KB-space if and only if no sublattice of $X$ is a copy of $c_{0}$
\cite[2.4.12, p. 92]{MN}, \cite[II.5.15, p. 95]{Sch}. It is evident from
Lozanovsky's Theorem that a reflexive Banach lattice is a KB-space.

\begin{theorem}
\label{t2} Let $X$ be a Banach space and let$\mathcal{\ B}$ be a Bade complete
Boolean algebra of projections on $X$ that is of finite multiplicity. Then the
conditions (1)-(5) of Theorem \ref{t1} are equivalent.

\begin{proof}
Once again, (1) implies (3), (3) implies (2), (2) implies (4), and (4) if and
only if (5) are clear. We need to show (4) implies (1). Suppose there exists a
positive integer $n$ such that $e_{m}=0$ for all $m>n$. Then, by the Corollary
\ref{c1}, $X$ is the direct sum of a finite collection of reflexive Banach
spaces. Hence $X$ is reflexive. Therefore, without loss of generality, assume
that $e_{n}\neq0$ for all $n$. Note that , by (4) and the Corollary \ref{c1},
we have that the subspace $e_{n}X$ is reflexive for each $n$. We need to show
that this extends to $X$ . For each $n$, let $\chi_{n}=e_{1}+e_{2}%
+\ldots+e_{n}\in\mathcal{B}$. It is clear that $\chi_{n}X$ is also reflexive.

Initially motivated by an idea of James \cite[Lemma 2]{J}, we will prove the
following: (*) Suppose that each cyclic subspace of $X$ does not contain any
copy of $l^{1}$, then $\parallel f-\chi_{n}^{\prime}f\parallel\rightarrow0$
for all $f\in X^{\prime}$.

For some $f\in X^{\prime}$ with $\Vert f\Vert=1$, suppose that (*) is false.
Since $\{\chi_{n}x\}$ converges to $x$ in norm for all $x\in X$, we have
$\{\chi_{n}^{\prime}f\}$ converges to $f$ in the weak*-topology in $X^{\prime
}$. Therefore that $\{\chi_{n}^{\prime}f\}$ does not converge to $f$ in norm
implies that $\{\chi_{n}^{\prime}f\}$ does not converge in norm in $X^{\prime
}$. That is, $\{\chi_{n}^{\prime}f\}$ is not a Cauchy sequence. Hence there
exists an $\varepsilon>0$ and a subsequence $\{\chi_{k(n)}\}$ such that
\[
\Vert\chi_{k(n+1)}^{\prime}f-\chi_{k(n)}^{\prime}f\Vert\geq\varepsilon
\]
for all $n$ and $\sup\chi_{k(n)}=1$. We define a disjoint sequence of
idempotents $\{\zeta_{n}\}$ in $\mathcal{B}$ with $\sup\zeta_{n}=1$ such that
(i) $\zeta_{1}=\chi_{k(2)}$, (ii) $\zeta_{n}=\chi_{k(n+1)}-\chi_{k(n)}$ for
$n\geq2$, and (iii) $\Vert\zeta_{n}^{\prime}f\Vert\geq\varepsilon$ for all
$n$. This implies that $X^{\prime}$ contains a copy of $l^{\infty}$. In fact
if $(\alpha_{n})\in l^{\infty}$, then, as defined above, $%
{\textstyle\sum}
\alpha_{n}\zeta_{n}\in C(K)$. Hence%
\[
\varepsilon|\alpha_{k}|\leq\Vert\alpha_{k}\zeta_{k}^{\prime}f\Vert=\Vert
\zeta_{k}^{\prime}(%
{\textstyle\sum}
\alpha_{n}\zeta_{n})^{\prime}f\Vert\leq\Vert(%
{\textstyle\sum}
\alpha_{n}\zeta_{n})^{\prime}f\Vert
\]
for each $k$ and
\[
\varepsilon\Vert(\alpha_{n})\Vert_{\infty}\leq\Vert(%
{\textstyle\sum}
\alpha_{n}\zeta_{n})^{\prime}f\Vert\leq\Vert(\alpha_{n})\Vert_{\infty}%
\]
for all $(\alpha_{n})\in l^{\infty}$. Clearly $\{\zeta_{n}^{\prime}f\}$
corresponds to the standard basis of $c_{0}$ and $f$ corresponds to the unit
of $l^{\infty}$. Now in a standard manner one can show that $X$ contains a
copy of $l^{1}$. Namely, for each $n$, we use (iii) to find $x_{n}\in\zeta
_{n}X$ such that $\Vert x_{n}\Vert=1$ and $f(x_{n})=\zeta_{n}^{\prime}%
f(x_{n})>\frac{\varepsilon}{2}$. Given any $(\xi_{n})\in l^{1}$, consider
$\xi_{n}=|\xi_{n}|e^{i\theta_{n}}$ for some $\theta_{n}\in\lbrack0,2\pi)$, for
all $n$. Then $(e^{-i\theta_{n}})\in l^{\infty}$ implies $%
{\textstyle\sum}
e^{-i\theta_{n}}\zeta_{n}\in C(K)$ and $(%
{\textstyle\sum}
e^{-i\theta_{k}}\zeta_{k})\xi_{n}x_{n}=|\xi_{n}|x_{n}$ in $X$ for each $n$.
Hence%
\[
(%
{\textstyle\sum}
e^{-i\theta_{k}}\zeta_{k})(%
{\textstyle\sum}
\xi_{n}x_{n})=%
{\textstyle\sum}
|\xi_{n}|x_{n}%
\]
and
\[
\Vert(%
{\textstyle\sum}
e^{-i\theta_{k}}\zeta_{k})(%
{\textstyle\sum}
\xi_{n}x_{n})\Vert\leq\Vert%
{\textstyle\sum}
\xi_{n}x_{n}\Vert
\]
in $X$. Here $%
{\textstyle\sum}
\xi_{n}x_{n}$, as before, denotes the limit of the partial sums of the series
in $X$. When we apply $f$ to both sides of the above equality, we have
\[
\frac{\varepsilon}{2}\Vert(\xi_{n})\Vert_{1}\leq%
{\textstyle\sum}
|\xi_{n}|f(x_{n})=f((%
{\textstyle\sum}
e^{-i\theta_{k}}\zeta_{k})(%
{\textstyle\sum}
\xi_{n}x_{n}))\leq\Vert%
{\textstyle\sum}
\xi_{n}x_{n}\Vert\leq\Vert(\xi_{n})\Vert_{1}.
\]
Let $y=%
{\textstyle\sum}
\frac{x_{n}}{2^{n}}\in X$. Then $\zeta_{n}y=\frac{x_{n}}{2^{n}}$ for each $n$.
Hence the subspace of $X$ that is a copy of $l^{1}$ is contained in the cyclic
subpace $X(y)$. But this is a contradiction. Therefore (*) is proved.

Now assume that $\{x_{\lambda}\}$ is an ultranet in the unit ball of $X$. Then
$\{e_{n}x_{\lambda}\}$ is an ultranet in the weakly compact unit ball of the
reflexive space $e_{n}X$. Therefore $\{e_{n}x_{\lambda}\}$ converges weakly to
some $y_{n}$ in the unit ball of $e_{n}X$. Let $u=%
{\textstyle\sum}
\frac{y_{n}}{2^{n}}$ and let $Y=X(u)$. Then $Y$ may be represented as a Banach
lattice with the quasi-interior point $u.$ Then (4) and Lozanovsky's Theorem
imply that $Y$ is a reflexive Banach lattice. Therefore $Y$ is a KB-space (see
the discussion preceeding the statement of the theorem). Let $z_{n}%
=y_{1}+y_{2}\ldots+y_{n}\in\chi_{n}X$, for each $n$. Then the ultranet
$\{\chi_{n}x_{\lambda}\}$ converges weakly to $z_{n}$. Then $z_{n}$ must be in
the unit ball of $\chi_{n}X$ since $\{\chi_{n}x_{\lambda}\}$ is in the unit
ball of $\chi_{n}X$. But this means that $\{z_{n}\}$ is a positive, increasing
sequence in the unit ball of the KB-space $Y$, since $e_{n}u=\frac{y_{n}%
}{2^{n}}$ for each $n$, implies that $\{y_{n}\}$ is a positive sequence in
$Y$. Therefore there exists $z\in Y$ such that $z=\sup z_{n}$ and $\{z_{n}\}$
converges to $z$ in norm in the unit ball of $Y$. Hence, given $\varepsilon
>0$, we have%
\[
\Vert z-z_{n}\Vert<\frac{\varepsilon}{4}%
\]
for sufficiently large $n$. Also for any $f\in X^{\prime}$ with $\Vert
f\Vert=1$, by (*), we have%
\[
\Vert f-\chi_{n}^{\prime}f\Vert<\frac{\varepsilon}{4}%
\]
for sufficiently large $n$. Consider%
\[
f(x_{\lambda}-z)=(f-\chi_{n}^{\prime}f)(x_{\lambda}-z)+\chi_{n}^{\prime
}f(z_{n}-z)+\chi_{n}^{\prime}f\left(  x_{\lambda}-z_{n}\right)
\]
for all $\lambda$ and for all $n$. Then for some fixed $n$ that is
sufficiently large, we have
\[
|f(x_{\lambda}-z)|<\frac{\varepsilon}{2}+\frac{\varepsilon}{4}+|\chi
_{n}^{\prime}f(\chi_{n}x_{\lambda}-z_{n})|
\]
for all $\lambda$. Since $\{\chi_{n}x_{\lambda}\}$ converges weakly to $z_{n}$
in $\chi_{n}X$, it follows that the ultranet $\{x_{\lambda}\}$ converges to
$z$ weakly in the unit ball of $X$. But any net in $X$ has a subnet that is an
ultranet. That is we have proved that any net in the unit ball of $X$ has a
weakly convergent subnet with limit point in the unit ball. Hence, the unit
ball of $X$ is weakly compact and $X$ is reflexive.
\end{proof}
\end{theorem}

Dieudonn\'{e}~\cite{D} constructed the famous example of a Banach space $X$
and a Bade complete Boolean algebra of projections $\mathcal{B}$ on $X$ that
is of uniform multiplicity 2 and has the following property $(\mathcal{D})$:
for any $x,y \in X$ and any $e \in\mathcal{B} \setminus\{0\}$, $eX$ is not
equal (or even isomorphic) to the sum of the cyclic subspaces $eX(x)$ and
$eX(y)$. The space $X$ in Dieudonn\'{e}'s example is not reflexive but a minor
modification of his example outlined in Example~\ref{e3} below provides a
reflexive space with similar properties.

\begin{example}
\label{e3} (Dieudonn\'{e}) Let the interval $[0,\gamma]$ and the functions
$\omega_{i},\ i=1,2,3$, be as constructed in~\cite[Section 6]{D}. Let $f$ be a
Lebesgue measurable function on $[0,\gamma]$. As usual we denote the
equimeasurable decreasing rearrangement of $|f|$ as $|f|^{\star}$. Let
$L(0,\gamma)$ be the space of equivalence classes of almost everywhere finite
measurable functions on $[0,\gamma]$. It is familiar that $L(0,\gamma)$ is a
Dedekind complete vector lattice and has $1,$ the unit of $L^{\infty}%
(0,\gamma)$, as a weak order unit. As such multiplication by the functions in
$L^{\infty}(0,\gamma)$ induces on $L(0,\gamma)$ a module structure.
Furthermore any (order) ideal in $L(0,\gamma)$ inherits the same module
structure. In fact a subspace of $L(0,\gamma)$ is an ideal if and only if it
is an $L^{\infty}(0,\gamma)$-submodule of $L(0,\gamma)$. Consider the Lorentz
spaces
\[
L_{\omega_{i}}^{2}=\{f\in L(0,\gamma):\;\int\limits_{0}^{\gamma}(|f|^{\star
})^{2}\omega_{i}dx<\infty\},\ i=1,2,3
\]
with the norm $N_{i}(f)=\big{(}\int\limits_{0}^{\gamma}(|f|^{\star})^{2}%
\omega_{i}dx\big{)}^{1/2}$.

The spaces $L_{\omega_{i}}^{2},\ i=1,2,3$, are ideals in $L(0,\gamma)$ and are
reflexive (see~\cite{Lo}) Banach lattices with the ideal center $Z(L_{\omega
_{i}}^{2})=L^{\infty}(0,\gamma)$, $i=1,2,3$. Dieudonn\'{e}'s construction
in~\cite{D} shows that $\omega_{i}\omega_{j}\in L^{1}(0,\gamma)$ if and only
if $i\neq j,\ (i,j=1,2,3)$. Moreover, $\omega_{i}^{1/2}\in L_{\omega_{j}}^{2}$
if and only if $i\neq j,(i,j=1,2,3)$. Consider the Banach $L^{\infty}%
(0,\gamma)$-module $E=L_{\omega_{1}}^{2}\oplus L_{\omega_{2}}^{2}\oplus
L_{\omega_{3}}^{2}$ endowed with the norm $N(f,g,h)=N_{1}(f)+N_{2}%
(g)+N_{3}(h)$. Let $X=\{(f,g,h)\in E:\;f+g+h=0\}$ and let $\mathcal{B}$ be the
Boolean algebra of all the idempotents in $L^{\infty}(0,\gamma)$. Then the
proof given by Dieudonn\'{e} in~\cite{D} shows that $\mathcal{B}$ is a Bade
complete Boolean algebra of projections on $X$ that is of uniform multiplicity
2 with the property $\mathcal{D}$.
\end{example}

\begin{remark}
\label{r1} In his study of multiplicity of Boolean algebras of projections,
motivated by Dieudonn\'{e}'s Example, Tzafriri (see~\cite{Tz1}) gave the
following formal definition of property $\mathcal{D}$. Suppose $X$ is a Banach
space and $\mathcal{B}$ is a Bade complete Boolean algebra of projections on
$X$ of uniform multiplicity $n$.

$\mathcal{B}$ has property $\mathcal{D}$ : for any $x_{i} \in X (i=1, \ldots,
n)$, any $e \in\mathcal{B} \setminus\{0\} $, and any $p, \ 1 \leq p < n$, $eX$
is not equal to the sum of $eX(x_{1}, \ldots, x_{p})$ and $eX(x_{p+1}, \ldots,
x_{n})$.

Tzafriri showed in~\cite{Tz1} that $\mathcal{B}$ has property $\mathcal{D}$ on
$X$ if and only if any bounded projection on $X$ that commutes with
$\mathcal{B}$ is itself in $\mathcal{B}$.
\end{remark}

In connection with Example~\ref{e3} one can consider the following question.

\begin{problem}
\label{p1} When is it possible to embed a reflexive (in particular, finitely
generated) Banach $C(K)$-module into a reflexive Banach lattice as a closed subspace?
\end{problem}

\centerline{ \textbf{Acknowledgements}}

We are grateful to H. Rosenthal and T. Oikhberg, respectively, for remarks
that allowed us to simplify condition (3) of Theorem~\ref{t1} and the proof of
Lemma~\ref{L3}, respectively.

\end{document}